\documentclass[reqno]{amsart}

\usepackage{amsmath}
\usepackage{amssymb}
\usepackage{amsfonts}
\usepackage{graphicx}
\usepackage{amsthm}
\usepackage{enumerate}
\usepackage{lscape}
\usepackage{dsfont}
\usepackage{color}
\usepackage{mathtools}

\usepackage{setspace}
\newcommand{\R}{\mathds{R}}                   
\newcommand{\z}{\mathds{Z}}

\newcommand{\CP}{\mathds{C}\mathrm{P}}
\newcommand{\N}{\mathds{N}}

\newcommand{\CH}{\mathds{C}\mathrm{H}}

\newcommand{\ol}{\mathrm{Hol}}
\newcommand{\Ric}{\mathrm{Ric}}

\newcommand{\C}{\mathds{C}}            
\newcommand{\de}{\partial}          

\newcommand{\K}{K\"{a}hler}

\newcommand{\ov}[1]{\overline{#1}}

\newcommand{\Sz}{Szeg\"{o}}
\renewcommand{\S}{\mathcal{S}}

\def\b{\beta}

\def\bC{{\mathbb C}}

\def\b1{{\rm id}}

\newtheorem{theor}{Theorem}[section]
\newtheorem{prop}[theor]{Proposition}

\newtheorem{lem}[theor]{Lemma}

\newtheorem{remark}{Remark}

\begin{document}

\title[Finite TYCZ expansions and cscK metrics ]{Finite TYCZ expansions and cscK metrics}

\author{Andrea Loi}
\address{(Andrea Loi) Dipartimento di Matematica \\
         Universit\`a di Cagliari (Italy)}
         \email{loi@unica.it}

\author{Roberto Mossa}
\address{(Roberto Mossa) 
Departamento de Matem\'atica \\
         Universidade Federal de Santa Catarina (Brasil)}
         \email{roberto.mossa@gmail.com}

\author{Fabio Zuddas}
\address{(Fabio Zuddas) Dipartimento di Matematica e Informatica \\
           Universit\`a di Cagliari (Italy)}
         \email{fabio.zuddas@unica.it}

\thanks{
The first two  authors were  supported  by Prin 2015 -- Real and Complex Manifolds; Geometry, Topology and Harmonic Analysis -- Italy, by INdAM. GNSAGA - Gruppo Nazionale per le Strutture Algebriche, Geometriche e le loro Applicazioni,  by GESTA - Funded by Fondazione di Sardegna and Regione Autonoma della Sardegna and by KASBA- Funded by Regione Autonoma della Sardegna.}

\begin{abstract}
Let $(M, g)$ be  a  \K\ manifold whose associated  \K\ form $\omega$ is integral and let $(L, h)\rightarrow (M, \omega)$
be a quantization hermitian line  
 bundle.
In  this paper we study  those \K\ manifolds $(M, g)$ admitting a {\em finite} TYCZ expansion, namely  those for which the associated Kempf  distortion function $T_{mg}$ is of the form: 
$$T_{mg}(p)= f_{s}(p) m^{s}+f_{s-1}(p) m^{s-1}+\dots+f_{r}(p)m^{r}, \quad f_{j}\in \mathcal{C}^{\infty}(M), \ 
s,r\in \z.$$ 
We show that if  the TYCZ expansion is finite then $T_{mg}$ is indeed a polynomial in $m$ of degree $n$, $n=\dim_\C M$, and the log-term of the  Szeg\"{o} kernel of the disc bundle $D\subset L^*$ 
vanishes
(where $L^*$ is the dual bundle of $L$).  Moreover, we provide  a complete classification of the \K\ manifolds  admitting  finite TYCZ expansion either  when $M$ is a complex curve or  when $M$ is a complex surface with a cscK metric which admits a   radial  \K\ potential.
\end{abstract}
 
\maketitle

\tableofcontents  

\section{Introduction}

Let $M$ be a  (not necessarily compact) $n$-dimensional complex manifold endowed with a K\"ahler metric $g$. Assume that there exists a 
holomorphic line bundle  $L$ over $M$ such that   $c_1(L)=[ \omega ]$,
where $\omega$ is the K\"{a}hler form associated to $g$ and $c_1(L)$
denotes the first Chern class of $L$
(such an $L$ exists if and only if  $\omega$ is an integral form).
Let $m\geq 1$ be an
integer and let $h_m$ be an  Hermitian metric on
$L^m=L^{\otimes m}$ such that its Ricci curvature ${\rm Ric}
(h_m)=m\omega$. 
Here ${\rm Ric} (h_m)$ is the two--form on $M$ whose
local expression is given by
\begin{equation}\label{rich}
{\rm Ric} (h_m)=-\frac{i}{2\pi}
\partial\bar\partial\log h_m(\sigma (x), \sigma (x)),
\end{equation}
for a trivializing holomorphic section $\sigma :U\rightarrow
L^m\setminus\{0\}$. 
In the quantum mechanics terminology $L^m$ is
called the {\em prequantum line bundle},
 the pair $(L^m, h_m)$ is called a {\em geometric
quantization} of the K\"{a}hler manifold $(M, m\omega)$ and $\hbar =
m^{-1}$ plays the role of Planck's constant  (see e.g.
\cite{arlquant}). Consider the separable complex Hilbert space
$\mathcal{H}_m$ consisting of global holomorphic sections  $s$ of $L^m$
such that
$$\langle s, s\rangle_m=
\int_Mh_m(s(p), s(p))\frac{\omega^n}{n!}<\infty .$$
Define the {\em Kempf distortion function}\footnote{In the literature the function $T_{m g}$ was first introduced under the name of $\eta$-{\em function} by J. Rawnsley in \cite{rawnsley}, later renamed as $\theta$-{\em function} in \cite{cgr1} followed by the {\em distortion function } of G. R. Kempf \cite{ke} and S. Ji \cite{ji}, for the special case of Abelian varieties and of S. Zhang
\cite{zha} for complex projective varieties.}, namely the smooth function on $M$ defined by:
\begin{equation}\label{Tmo}
T_{m g} (p) =\sum_{j=0}^{N(m)}h_m(s_j(p), s_j(p)),
\end{equation}
where $s_j$, $j=0, \dots , N(m)$ ($\dim \mathcal{H}_m=N(m)+1\leq\infty$) is an orthonormal basis of $\mathcal{H}_m$.

As suggested by the notation  this function depends only on the metric $m g$ and not on   the orthonormal basis chosen. Obviously if $M$ is compact $\mathcal{H}_m=H^0(L^m)$, where 
$H^0(L^m)$ is the (finite dimensional) space of global holomorphic sections of $L^{m}$.

By applying the methods developed in \cite{BoutSjos} and specifically the parametrix for the \Sz\ kernel, D. Catlin \cite{cat} and S. Zelditch
\cite{Zel} independently 
proved that if in the above setting $M$ is compact,  there exists
a complete asymptotic expansion  of the  Kempf distortion function:
\begin{equation}\label{asymptoticZ}
T_{mg}(p) \sim \sum_{j=0}^\infty  a_j(p)m^{n-j},
\end{equation}
where $a_0(p)=1$ and $a_j(p)$, $j=1,\dots$ are smooth functions on $M$.
This means that, for any nonnegative integers $r,k$ the following estimate holds:
\begin{equation}\label{rest}
||T_{mg}(p)-
\sum_{j=0}^{k}a_j(p)m^{n-j}||_{C^r}\leq C_{k, r}m^{n-k-1},
\end{equation}   where $C_{k, r}$ are constant depending on $k, r$ and on the
K\"{a}hler form $\omega$ and $ || \cdot ||_{C^r}$ denotes  the $C^r$
norm.
The expansion \eqref{asymptoticZ} is called {\em Tian--Yau---Catlin-Zelditch expansion} (TYCZ expansion in the sequel). 
Later on,  Z. Lu \cite{lu},  by means of  Tian's peak section method,
 proved  that each of the coefficients $a_j(p)$ 
is a polynomial
of the curvature and its
covariant derivatives at $p$ of the metric $g$ which can be found
 by finitely many algebraic operations.
In particular, he computed the first three coefficients.
The first two are given by:
\begin{equation}\label{coefflu}
\left\{\begin{array}
{l}
a_1(p)=\frac{1}{2}{\rm scal}_g\\
a_2(p)=\frac{1}{3}\Delta{\rm scal}_g
+\frac{1}{24}(|R|^2-4|{\rm Ric} |^2+3{\rm scal}_g ^2),
\end{array}\right.
\end{equation}
where ${\rm scal}_g$, ${\rm Ric}$, $R$, are, respectively,  the scalar curvature, the Ricci tensor and the Riemann curvature tensor
of  $(M, g)$, in local coordinates.
The reader is also referred to  
\cite{loianal} and \cite{asymsmooth} for a  recursive formula  for the coefficients $a_j$'s  and  an alternative computation of  $a_j$ for $j\leq 3$ using Calabi's diastasis function  (see also \cite{LZedda15} for the case of locally Hermitian symmetric spaces).
When $M$ is noncompact, there is not a general theorem which assures the existence of an asymptotic expansion \eqref{asymptoticZ}. 
Observe that in this case  we say that an asymptotic expansion \eqref{asymptoticZ} exists if \eqref{rest} holds for any compact subset of  $M$. M. Engli\v{s} \cite{englis2} showed that a TYCZ  expansion exists in the case of strongly pseudoconvex bounded domains in $\C^n$ with real analytic boundary, and proved that the first three coefficients are the same as those computed by Lu for compact manifolds.  
The reader is referred to  \cite{MaMarinescu}  for the description of some curvature conditions which assure the existence of a TYCZ expansion in the noncompact case (see also \cite{commtodor} and  \cite{LoiZeddaTaub} for some explicit examples).

Consider the negative Hermitian line bundle $(L^*,h^* )$ over $(M, g)$ dual   to $(L, h)$ and
let $D\subset L^*$ be the unit disk bundle over $M$, i.e.
\begin{equation}
D=\{v\in L^* \ |\  \rho (v):=1-h^*(v, v)>0\}.
\end{equation}
It is not hard to see (and well-known)  that the condition $\Ric (h)=\omega$ implies that $D$ is a strongly pseudoconvex   domain in $L^*$ with smooth boundary
$X=\partial D=\{v\in L^*\ |\ \rho (v)=0\}$. $X$ will be called  the  {\em unit circle bundle}. 
Let ${\mathcal S}(v)$ be  the {\it Szeg\"{o}  kernel} of $D$ (see Section \ref{sec1} below).
By a fundamental result due to   Boutet de Monvel and Sj\"{o}strand \cite{BoutSjos} \footnote{This formula \eqref{fefferman} has been proved for strictly pseudoconvex  complex domains  in $\C^n$ with smooth boundary, but it could be easily extended to the disc bundle $D\subset L^*$ (see, e.g., \cite{LuTian}).}
there exist $a, b\in C^{\infty} (\bar D)$, $a\neq 0$ on $X$ such that:
\begin{equation}\label{fefferman}
{\mathcal S}(v)=a(v)\rho(v)^{-n-1}+b(v)\log\rho (v), \ v\in D.
\end{equation}
The function $b(v)\log\rho (v)$ in (\ref{fefferman}) is called the {\em logarithmic term}  ({\em log-term} from now on) of the Szeg\"{o} kernel.  One says that the log-term  of the Szeg\"{o} kernel of the disk  bundle $D\subset L^*$ vanishes if $b=0$ identically on $D$.
The  Szeg\"{o} kernel  is strictly related to the  Kempf distortion function.
Indeed Z. Lu and G. Tian  \cite{LuTian}  prove that\footnote{The proof is given in the compact setting but it is of local nature so it  immediately extends to the noncompact one.} if  the  log-term of the disk bundle $D\subset L^*$ vanishes then
$a_k=0$ for $k>n$, where $a_k$ are the coefficients appearing in \eqref{asymptoticZ}. A conjecture still open, due to a private communication with  Z. Lu, asks if the vanishing of the $a_k$'s for $k>n$ implies the vanishing of the log-term.

In this paper we address the problem of studying those \K\ manifolds whose TYCZ expansion is {\em finite}, namely  the Kempf distortion function is of the form:
\begin{equation}\label{TPoly}
T_{mg}(p)= f_{s}(p) m^{s}+f_{s-1}(p) m^{s-1}+\dots+f_{r}(p)m^{r}, \quad f_{j}\in \mathcal{C}^{\infty}(M), \ 
s,r\in \z.
\end{equation}
Notice that this sort of  problem has been  partially investigated in the compact setting  by the first author of the present paper and by  C. Arezzo \cite{arlquant}.

One can give a
quantum-geometric interpretation of $T_{mg}$   as follows.
Assume that there exists $m$ sufficiently large such  that for each point $x\in M$
there exists $s\in \mathcal{H}_m$ non-vanishing at $x$ (such an $m$ exists if $M$ is compact  by standard algebraic geometry methods
and corresponds to the free-based point condition in Kodaira's theory).  Consider the
so called  {\em coherent states map}, namely the
holomorphic map of $M$ into the complex projective space
${\C}P^{N(m)}$ given by:
\begin{equation}\label{psiglob}
\varphi_m :M\rightarrow {\C}P^{N(m)}:
x\mapsto [s_0(x): \dots :s_{N(m)}(x)].
\end{equation}
One can prove (see, e.g.  \cite{arlcomm}) that
\begin{equation}\label{obstr}
\varphi ^*_m\omega_{FS}=m\omega_g +
\frac{i}{2\pi}\partial\bar\partial\log T_{mg} ,
\end{equation}
where $\omega_{FS}$ is the Fubini--Study form on
${\C}P^{N(m)}$, namely the  \K\ form which in homogeneous
coordinates $[Z_0,\dots, Z_{N(m)}]$ reads as 
$\omega_{FS}=\frac{i}{2\pi}\partial\bar\partial\log \sum_{j=0}^{N(m)}
|Z_j|^2$.
Recall that a K\"ahler metric $g$ on a complex manifold $M$ is said to be \emph{projectively induced} if there exists a K\"ahler (isometric and holomorphic) immersion of $(M, g)$ into the finite or infinite dimensional   complex projective space $(\C P^N, g_{FS})$, $N \leq +\infty$,  endowed with the Fubini--Study metric $g_{FS}$. The reader is referred to \cite{LoiZedda-book} for further details and for  un updated account on projectively induced \K\ metrics.
Obviously not all K\"{a}hler metrics are projectively induced.
Nevertheless, by combining  \eqref{obstr} and the existence of a  TYCZ expansion one gets that 
$\frac{\varphi_m ^*g_{FS}}{m}$ $C^{\infty}$-converges to $g$.
 In other words, any 
metric $g$ with integral \K\ form $\omega$  on a  complex manifold is the $C^{\infty}$-limit of
(normalized) projectively induced K\"{a}hler metrics (under the assumption 
of the existence of a TYCZ expansion). 
In the compact case this was a conjecture of Yau   proved by 
G. Tian \cite{ti0} and  W. D. Ruan \cite{ru} by means of peak section method.

The following theorem represents our first result.

\begin{theor}\label{mainteor}
Let $(M, g)$ be a \K\ manifold with integral \K\ form $\omega$ and of complex dimension $n$.
Assume that  the corresponding  TYCZ expansion is finite.
Then $T_{mg}(p)$ is forced to be a  polynomial in $m$ of degree $n$ and the log-term of the Szeg\"{o} kernel of the disc bundle $D$ vanishes.
\end{theor}

The concept of finite TYCZ expansion is strictly related to  {\em regular quantizations} introduced in \cite{cgr1} in the context of the quantization by deformation of \K\ manifolds. 
One says that the quantization $(L, h)$ of a \K\ manifold $(M, g)$ is {\em regular} if the   Kempf distortion function
$T_{mg}$ (exists and) is a strictly positive constant for all  $m$ sufficiently large (see also \cite{berloimossa} and  \cite{FCAL} and reference therein). In S. Donaldson \cite{do} terminology a \K\ metric $g$ with integral \K\ form $\omega$ such that its Kempf distortion function is a positive constant is called {\em balanced}. Hence a quantization of a \K\ manifold $(M, g)$ is regular if  
$mg$ is balanced for all $m$ sufficiently large.

When $(M, g)$ is a compact  \K\ manifold which admits a regular quantization
then the TYCZ expansion is necessarily finite.
Indeed in that case $T_{mg}=\frac{h^{0}(L^m)}{V(M)}$, where $h^0(L^m)$ denotes the complex dimension of $H^0(L^m)$
and $V(M)=\int_M\frac{\omega^n}{n!}$ is the volume of $M$,
and so by Riemann--Roch theorem  $T_{mg}$ is a  monic  polynomial in  $m$  of degree $n$.
Thus, the vanishing of the log-term of the Szeg\"{o} kernel  in the last part of  Theorem \ref{mainteor} (which is in 
accordance with the above mentioned Lu's conjecture) extends the results obtained in \cite{ALZ}  and \cite{LMZ} in the compact and regular case. We believe that  in the compact case, finite TYCZ expansion implies regular quantization.

Nevertheless, in the noncompact case there exist \K\ manifolds with nonconstant Kempf distortion function and finite TYCZ expansion.
In order to describe an example  assume that  $M$ is  a complex domain (open and connected) of  $\C^n$ 
equipped with a  global \K\ potential  $\Phi :M\rightarrow\R$, 
i.e. $\omega=\frac{i}{2\pi}\partial\bar\partial \Phi$. In this case  $\omega$ is trivially integral and  the Hilbert space $\mathcal{H}_m$
agrees with $\mathcal{H}_{m\Phi}$
the weighted Hilbert space of square integrable holomorphic functions on $M$, with weight $e^{-m\Phi}$, namely
\begin{equation}\label{hilbertspace}
\mathcal{H}_{m\Phi}=\left\{ f\in\ol(M) \ | \ \, \int_M e^{-m\Phi}|f|^2\frac{\omega^n}{n!}<\infty\right\}.
\end{equation}
If $\mathcal{H}_{m\Phi}\neq \{0\}$
then   the Kempf distortion function reads as
\begin{equation}\label{kempfnonc}
T_{mg}(z)=e^{-m\Phi(z)}K_{m\Phi}(z, z), 
\end{equation}  
where  $K_{m\Phi}(z, z)=\sum_{j=0}^{N(m)} |f_j(z)|^2$ is the weighted reproducing kernel 
and $\{f_j\}$ an orthonormal basis  for $\mathcal{H}_{m\Phi}$.
Let now $p$ be a positive real number and 
$$M=\{(z_1, z_2)\in\C^2 \ | \ |z_1|^2+|z_2|^{\frac{2}{p}}<1\}$$
equipped with the \K\ form $\omega$, with \K\ potential
$$\Phi =-\log \left[(1-|z_1|^2)^p-|z_2|^2\right].$$
A straightforward computation  (see e.g. \cite[pp. 450-451]{me}) shows that the weighted reproducing kernel
is given by
$$K_{m\Phi}(z, z)=e^{m\Phi}\left[m^2+(c(z)-3)m+c(z)+2\right]$$
where
$$c(z)=\left(1-\frac{1}{p}\right)\left(1-\frac{|z_2|^{2}}{(1-|z_1|^2)^p}\right).$$
Thus, by \eqref{kempfnonc}, the Kempf distortion function reads as
\begin{equation}\label{exfinitenon}
T_{mg}(z)=m^2+(c(z)-3)m+c(z)+2.
\end{equation}
It follows that  for $p\neq 1$, $T_{mg}$ is a polynomial  in $m$ of degree $2$
with nonconstant coefficients ($a_1(z)=c(z)-3$ and $a_2(z)=c(z)+2$).
Notice that for $p=1$, $M$ is the complex hyperbolic plane, namely the   unit ball in $\C^2$
and $\omega$ equals the hyperbolic form,  and in this case the quantization is regular (see also below).

Our second result shows that, for a complex curve, finite TYCZ expansion  implies regular quantization and that this happens
only in  the complex space form case.

\begin{theor}\label{class1}
Let $M$ be a complex curve which admits a complete   \K\ metric $g$ whose corresponding  TYCZ expansion is finite. Then $(M, g)$ is \K\ equivalent to one of the following complex space forms:
\begin{itemize}
\item [(a)]
$(\C, g_0)$, where $g_0$ is the flat metric on $\C$.
\item [(b)]
$(\CH^1, \mu g_{hyp})$, where $g_{hyp}$ is the hyperbolic metric on the unit disk of $\C$ and $\mu$ is a positive real number.
\item [(c)]
$(\CP^1, \lambda g_{FS})$, where $g_{FS}$ is the Fubini-Study metric and $\lambda$ is a  positive integer.
\end{itemize}
\end{theor}

Many examples of  \K\ manifolds admitting regular quantizations are obtained by taking simply-connected homogeneous 
\K\ manifolds with integral \K\ forms (see \cite{arlcomm}). 
Hence, for example the complex space forms namely the flat space $(\C^n, g_0)$ with the flat \K\ form 
$\omega_0=\frac{i}{2\pi}\partial\bar\partial |z|^2$, the hyperbolic space $(\CH^n, g_{hyp})$, i.e. the unit ball in $\C^n$
with the hyperbolic form $\omega_{hyp}=-\frac{i}{2\pi}\partial\bar\partial\log(1- |z|^2)$,
 the complex projective space $(\CP^n, g_{FS})$, admit  regular quantizations which, as  one can easily verify,   have  finite TYCZ expansion.

While in the compact case the homogeneous \K\ manifolds   are the only known examples admitting a regular quantization, in the noncompact case the first author
together with F. C. Aghedu \cite{FCAL} prove that  the  Kempf distortion function for  the Simanca metric $g_S$ on the blow-up $\tilde\C^2$  of $\C^2$ at the origin  is given by  $T_{mg_S}=m^2$.
Hence,   for the Simanca metric  the quantization is not just  regular but the TYCZ is finite with constant coefficients (all the coeffcients $a_k=0$ for $k\geq 1$). Notice that, if $H$ denotes  the exceptional divisor,  then $g_S$
has radial K\"ahler potential  on the dense subset $U=\tilde\C^2\setminus H=\C^2\setminus\{0\}$ given by
\begin{equation}\label{simancaprimavolta}
\Phi(z)=|z_1|^2+|z_2|^2+\log (|z_1|^2+|z_2|^2).
\end{equation}

Our third and last  result  shows that the complete  complex surfaces with a cscK (\K\ with constant scalar curvature)  metric with densily defined radial potential and finite TYCZ expansion are essentially the complex space forms
and $(\tilde\C^2, g_S)$.

\begin{theor}\label{class2}
Let $M$ be a complex  surface which admits a complete   cscK metric $g$  whose corresponding  TYCZ expansion is finite.  Assume, moreover, that the metric $g$ admits a radial  \K\ potential $\Phi :U\rightarrow \R$ defined on a dense subset $U$ of $M$. Then $(M, g)$ is \K\ equivalent to one of the following \K\ surfaces:
\begin{itemize}
\item [(i)]
$(\C^2, g_0)$, where $g_0$ is the flat metric on $\C^2$.
\item [(ii)]
$(\CH^2, \mu g_{hyp})$, where $g_{hyp}$ is the hyperbolic metric on the unit disk of $\C^2$ and $\mu$ is a positive real number.
\item [(iii)]
$(\CP^2, \lambda g_{FS})$, where $g_{FS}$ is the Fubini-Study metric and $\lambda$ is a positive integer.
\item [(iv)]
$(\tilde\C^2, \lambda g_{S})$, where  $\tilde\C^2$ denotes the  blow-up of $\C^2$ at the origin, $g_S$ the Simanca metric and $\lambda$ is a positive integer.
\end{itemize}
\end{theor}      
\begin{remark}\rm
The assumption on  the potential in Theorem \ref{class2} means that $U$ can be equipped with global complex coordinates $z_1$ and $z_2$ and $\Phi$ only depends on $|z_1|^2+|z_2|^2$.
Notice also that  $\Phi$  is not necessarily defined at the origin (see Remark \ref{remarkorigin} below for details).
\end{remark}

\begin{remark}\rm
If  we assume $M=\CP^2$ and  the finiteness of TYCZ expansion 
 then, by using the last part of Theorem \ref{mainteor}, one can get that $g=\lambda g_{FS}$ for some integer $\lambda$, without further assumptions (either on the curvature or on the potential).
 Indeed, a deep result due to Z. Lu and G. Tian \cite{LuTian} asserts that  an integral  \K\ form on $\CP^2$  such that the log-term of the disk bundle vanishes is an integral positive mutiple of  the Fubini-Study form.
 \end{remark}

The paper is organized as follows. Section \ref{sec1} and Section \ref{sec3} are dedicated to the proofs of Theorem \ref{mainteor} and Theorems \ref{class1}--\ref{class2} respectively.
The proof of the latter is based on the classification of radial cscK projectively induced metrics with $a_3=0$ given in 
Section \ref{sec2} (see Proposition \ref{radprojindcscka3=0}).

\section{The proof of Theorem \ref{mainteor}}\label{sec1}

Let $(M, g)$ be a \K\ manifold.  Assume that the \K\ form $\omega$ associated to $g$ is integral and let $(L, h)\rightarrow  (M, \omega)$ be a quantization bundle and $D\subset L^*$ be the corresponding disk bundle as in the introduction.
The proof of Theorem \ref{mainteor}  is based on the link between
the  Szeg\"{o} kernel of the disk bundle $D$ and the Kempf distortion function (see Equation \eqref{szegokempf} below) and  on the two subsequent lemmata (Lemma \ref{boundedszego} and Lemma \ref{lemmatec}). 
In order to obtain \eqref{szegokempf} 
let us denote by $\mathcal H^2(X)$ the space  of boundary values of holomorphic functions on $D$ that are square integrable on $X$ with respect to the measure $d \mu = \theta \wedge (d\theta)^n$, being $d \theta = \omega$. The Hardy space ${\mathcal H^2}(X)$ admits the Fourier decomposition into irreducible factors with respect to  the natural  $S^1$-action. Namely, 
\begin{equation}\label{decom}
\mathcal H^2(X) = \bigoplus_{m=0}^{+ \infty} \mathcal H_m^2(X)
\end{equation}
where 
\begin{equation}\label{decompH2k}
\mathcal H_m^2(X) = \{ f \in \mathcal H^2(X) \ | \ f(e^{i \theta} x) = e^{i m \theta} f(x) \},
\end{equation}
equivalently, $f(\alpha v) = \alpha^m f(v)$ for $\alpha \in \C$ (since $f$ is holomorphic).

By definition the {\it Szego kernel} ${\mathcal S}(z,w)$ is the reproducing kernel of $\mathcal H^2(X)$, i.e. is characterized by the properties ${\mathcal S}(z,w) \in \mathcal H^2(X)$ for every fixed $w \in D$, ${\mathcal S}(w,z) = \overline{{\mathcal S}(z,w)}$ and
\begin{equation}\label{propcharSK}
f(z) = \int_X {\mathcal S}(z,w) f(w) d \mu_w
\end{equation}
for every $f \in \mathcal H^2(X)$  and $z \in D$.
From these properties it is immediately seen that ${\mathcal S}(z,w) = \sum_{j=1}^{\infty} f_j(z) \overline{f_j(w)}$, where $\{ f_j \}$ is an orthonormal basis of $\mathcal H^2(X)$.
Let us denote ${\mathcal S}(z) := {\mathcal S}(z,z) = \sum_{j=1}^{\infty} |f_j(z)|^2$.
Now, by (\ref{decom}), an orthonormal basis of $\mathcal H^2(X)$ can be obtained by putting together orthonormal bases of $\mathcal H_m^2(X)$ for $m = 0,1, \dots$. If $f_{1}, \dots, f_{{N(m)}}$ form an orthonormal basis of $\mathcal H_m^2(X)$, let us denote

\begin{equation}\label{defS_kconf}
{\mathcal S}_m(v) := \sum_{j=1}^{N(m)} |f_{j}(v)|^2
\end{equation}

Then, we can write

\begin{equation}\label{S=sumS_kconf}
{\mathcal S}(v) = \sum_{m=0}^{\infty} {\mathcal S}_m(v)
\end{equation}

\begin{remark}\label{Remark k=0}
\rm Notice that $\mathcal H_0^2(X)$ is the space of holomorphic functions $f$ on $D$ such that $f(e^{i \theta} x) = f(x)$ for every $x \in X$, i.e. the functions which are  constant on the fiber above every point $p \in M$ and square integrable on $X$. If $M$ is compact, $\mathcal H_0^2(X)$ obviously contains only the constant functions and a basis is given by $f \equiv c$ such that $\int_{X} |c|^2 d \mu = 1$. If $M$ is not compact, $\mathcal H_0^2(X)$ identifies with the space of holomorphic functions $f$ on $M$ such that $\int_M |f|^2 \omega^n < \infty$, and ${\mathcal S}_0(v)$ is constant on each fiber of $D$, i.e. it can be identified with a smooth function $F_0: M \rightarrow \C$.
\end{remark} 
Let $\mathcal{H}_m$ be the space of $L^2$-bounded holomorphic sections of $L^m$ defined in the introduction.
It is easy to see (see e.g. \cite{Zel} for the compact case) that for $m \geq 1$ there is a unitary equivalence $
 \mathcal{H}_m\rightarrow \mathcal H_m^2(X)$ which sends a section $s \in  \mathcal{H}_m$  to the function $\hat s \in \mathcal H_m^2(X)$ defined by
$$\hat s(\lambda) = \lambda^m(s)$$
for every $\lambda \in L^*$. Then, if we take an orthonormal basis $s_{1}, \dots, s_{{N(m)}}$ of $\mathcal{H}_m$ then $\hat s_{1}, \dots, \hat s_{{N(m)}}$ is an orthonormal basis of $\mathcal H_m^2(X)$.

Thus, for $m\geq 1$ we have 
$${\mathcal S}_m(v) = \sum_{j=1}^{N(m)} |\hat s_{j}(v)|^2= \sum_{j=1}^{N(m)} \left| (\sqrt{h^*(v,v)})^m \hat s_{j}\left( x  \right) \right|^2 = (h^*(v,v))^m \sum_{j=1}^{N(m)}  \left|  \hat s_{j}(x) \right|^2,$$
where we  denote by $x = \frac{v}{\sqrt{h^*(v,v)}}$.
Thus, since $\sum_{j=1}^{N(m)}  \left|  \hat s_{j}(x) \right|^2$ is  the Kempf distortion function 
$T_{mg}(\pi(x)) = T_{mg}(\pi(v))$ (where $\pi: L^* \rightarrow M$ is the bundle projection), we have
$${\mathcal S}_m(v) = (h^*(v,v))^m T_{mg}(\pi(v)).$$
Combining this with (\ref{S=sumS_kconf}), we can write\footnote{Equation \eqref{szegokempf}  extends  to the noncompact setting the analogous equation proved in \cite{ALZ}  for the  compact case.}
\begin{equation}\label{szegokempf} 
{\mathcal S}(v) = \sum_{m=0}^{\infty} (h^*(v,v))^m T_{mg}(p)
\end{equation}
for $v \in D$, where $p = \pi(v)$ and $T_0(p)$ is the function $F_0$ in Remark \ref{Remark k=0} above.
\begin{lem}\label{boundedszego}
Let $(M, g)$ be a \K\ manifold such that \K\ form $\omega$ associated to $g$ is integral.
Assume that the associated Kempf distortion function $T_{mg}$ 
admits a TYCZ expansion.
Let $p_{0} \in M$
and define
\begin{equation}\label{defphi}
\phi \left(t\right) =\sum _{m=0}^{\infty}  (1-t)^{n+1}t^{m}\,T_{mg}(p_0)
\end{equation}
Then the map $t \mapsto \phi^{(h)}(t):=\frac{\de^{h} }{\de t^{h}}\left(\phi(t)\right)$  is bounded on  $(0,1)$ for all $h\geq 0$.
\end{lem}
\proof
One has
\begin{equation*}
\begin{split}
\phi ^{(h)} \left(t\right)= &\sum _{m=0}^{\infty}  \left((1-t)^{n+1}t^{m} \right)^{(h)} T_{mg}(p_{0}) \\ 
=& \sum _{m=1}^{\infty}  \left((1-t)^{n+1}t^{m} \right)^{(h)} T_{mg}(p_{0})
+((1-t)^{n+1})^{(h)} T_{0}(p_{0})
\end{split}
\end{equation*}

On the other hand, by   \eqref{rest}, we have
$$
-{C_{0,0}}\,{m^{n-1}} +a_{0}\,  m^{n}
\leq T_{mg}(p_{0})
\leq {C_{0,0}}\,{m^{n-1}} + a_{0}\,  m^{n}, \ m\geq 1.
$$
Hence to show that $ \phi ^{(h)} \left(t\right) $ is bounded one needs to verify that the two functions
$$
\varphi_{k}(t)=\sum _{m=1}^{\infty}  \left((1-t)^{n+1}t^{m} \right)^{(h)}  {m^{k}},\ \ k=n-1,n
$$
are bounded on $(0,1)$. This easily follows since
\begin{equation*}
 \left(  \sum _{m=0}^{\infty}  (1-t)^{n+1}t^{m}   {m^{k}}\right)^{(h)} 
= \left(  q_{k}(t)(1-t)^{n-k}  \right)^{(h)},\  0<t<1,
\end{equation*}
where  $q_{k}(t)$ is the polynomial of degree $k$ in $t$
such that 
\begin{equation}\label{pol}
q_{k}(t)=(1-t)^{k+1}\sum _{m=0}^{\infty} t^{m}m^{k}, \ 0<t<1.
\end{equation}
\endproof
\begin{lem}\label{lemmatec}
Let $k_0$ be a positive integer and $h$ a natural number.
Consider the function
 \begin{equation}\label{psiht}
 \psi_{h}(t)=\left((1-t)^{n+1} \sum _{m=1}^{\infty}  \frac{t^{m}}{m^{k_{0}+h}} \right)^{(n+k_{0})},\ 0<t<1.
\end{equation}
Then $\psi_h(t)=O(1)$ in $[0, 1]$  if and only if $h\neq 0$.
 \end{lem}
\begin{proof}
Observe that 
\begin{align}
 \psi_{h}(t)&
 =  \sum _{l=k_0-1}^{n+k_{0}} c_l \left(1-t\right) ^{1-k_{0}+l}  \sum _{m=1}^{\infty}  \left(\frac{t^{m}}{m^{k_{0}+h}} \right)^{(l)} \nonumber \\
&= \sum _{l=k_{0}-1}^{n+k_{0}}\left(1-t\right) ^{1-k_{0}+l}\sum _{{m=l \text{ if } l\geq1 \atop m=1 \text{ if } l=0}}^{\infty} \left(b_{l,l}m^{l}+b_{l,l-1}m^{l-1}+\dots+b_{l,1}m+b_{l,0}\right)\frac{t^{m-l}}{m^{k_{0}+h}} \nonumber \\
& = \sum _{l=k_{0}-1}^{n+k_{0}}\sum _{s=0}^{l}b_{l,s}\sum _{m=l \text{ if } l\geq1 \atop m=1 \text{ if } l=0}^{\infty} \left( \frac{t^{m-l}}{m^{k_{0}+h-s}}\left(1-t\right) ^{1-k_{0}+l} \label{addendo}
\right)
 \end{align}
where $c_l,b_{l,l-1},\dots,b_{l,0},b_{l}$ are suitable real numbers. 
Consider  the series
\begin{equation}\label{Serie666}
F_{k_0, s, h, l}(t)=(1-t) ^{1-k_{0}+l}\sum _{m=l \text{ if } l\geq1 \atop m=1 \text{ if } l=0}^{\infty} \frac{t^{m-l}}{m^{k_{0}+h-s}}
\end{equation}
for $k_0-1 \leq l\leq n+k_{0}$ and $0 \leq s \leq l$.
Notice that for  $l=0$  (and hence $s=0$ and $k_0=1$) (\ref{Serie666}) reads as
$\sum_{m=1}^{\infty}\frac{t^m}{m^{h+1}}$
which is bounded for $t\rightarrow 1^-$ if and only if $h>0$.
More generally, we claim that 
(\ref{Serie666}) diverges if and only if $h=0$ and   $s=l=k_{0}-1$.

Indeed for $h=0$ and   $s=l=k_{0}-1$  (\ref{Serie666})   reads as  
 \begin{equation*}
 t^{-k_{0}+1} \sum _{m=k_{0}-1 \text{ if }  k_0\geq 2 \atop m=1 \text{ if } k_0=1}^{\infty}  \frac{t^{m}}{m} =t^{-k_{0}+1}  \left[\log \left(1-t\right) -  \sum _{m=1}^{k_{0}-2} \left( \frac{t^{m}}{m} \right)\right]
 \end{equation*}
and so it  tends to  $-\infty$ for $t \rightarrow 1^-$.
On the other hand for the other values of  the parameters 
one has 
the following case by case analysis which shows
that  (\ref{Serie666})  is bounded for $t \rightarrow 1^-$ 
(we assume $l\geq 1$ by the above considerations).

\noindent
{\bf Case 1.} $s>k_0+h$:

$F_{k_0, s, h, l}(t)=\left(1-t\right) ^{1-k_{0}+l}\sum _{m=l}^{\infty} \frac{t^{m-l}}{m^{k_{0}+h-s}}=(1-t) ^{1-k_{0}+l}  \sum _{\tilde m=0}^{\infty} {t^{\tilde m}}{(\tilde m+l)^{s-k_{0}-h}}$
 
$\ \ \ \ \ \ \ \ \ \ \ \ \ \ =  (1-t)^{l-s} \tilde q_{s-k_{0}-h}(t),$

where $ \tilde q_{s-k_{0}-h}(t)$  is a polynomial of degree $s-k_{0}-h$.

\noindent
{\bf Case 2.} $s=k_0+h$ (and hence $l\geq k_0$):

$F_{k_0, s, h, l}(t)=
\left(1-t\right) ^{1-k_{0}+l}t^{-l}\sum _{m=l}^{\infty} t^{m}=\left(1-t\right) ^{1-k_{0}+l}t^{-l}\left[\frac{1}{1-t}-\sum _{\tilde m=0}^{l-1} t^{\tilde m}\right]$

$\ \ \ \  \ \ \ \ \ \ \ \ \ \  =(1-t) ^{-k_{0}+l}t^{-l}-(1-t) ^{1-k_{0}+l}t^{-l}\sum _{\tilde m=0}^{l-1} t^{\tilde m}$

\noindent
{\bf Case 3.}  $s=k_0+h-1$ (and hence\footnote{Since  $l=k_0-1$ forces   $h=0$ and $s=l=k_0-1$.}
$l> k_0-1$):
$$F_{k_0, s, h, l}(t)=
\left(1-t\right) ^{1-k_{0}+l}t^{-l}\sum _{m=l}^{\infty} \frac{t^m}{m}=
\left(1-t\right) ^{1-k_{0}+l}t^{-l}\left[\log (1-t)-\sum _{\tilde m=1}^{l-1} \frac{t^{\tilde m}}{\tilde m}\right].
$$

\noindent
{\bf Case 4.}  $s\leq k_0+h-2$:
$$F_{k_0, s, h, l}(t)=
\left(1-t\right) ^{1-k_{0}+l}\sum _{m=l}^{\infty} \frac{t^{m-l}}{m^{k_{0}+h-s}}\leq
\left(1-t\right) ^{1-k_{0}+l}\sum _{m=l}^{\infty} \frac{t^{m-l}}{m^{2}}.
$$
\end{proof}

We can now  prove  Theorem \ref{mainteor}.
\begin{proof}[Proof of Theorem \ref{mainteor}]
We first prove that \eqref{TPoly} forces $T_{mg}(p)$  to be a polynomial of degree $n$. By (\ref{rest}) for $k=0$ and 
\eqref{TPoly} one gets
$$\left| \sum_{h=r}^{s} f_{h} \left(p\right) m^{h-n} - a_{0}\right|\leq {C_{0,0}}\,{m^{-1}},$$
and taking $m \rightarrow \infty$ one deduces $f_{n+1}= f_{n+2}= \dots f_{s}= 0$ and $f_{n}= a_{0}=1$.
It remains to show that  $r \geq 0$. Assume by a contradiction that $r<0$. Then 
 the function $\phi(t)$ given in (\ref{defphi}) decomposes as

\begin{equation}\label{decompPhi}
\phi(t)  = (1-t)^{n+1} T_0(p) + g_+(t) + g_-(t)
\end{equation}

where

$$g_+(t) := (1-t)^{n+1}\sum _{m=1}^{\infty}   t^{m} \left(f_{n}(p) m^n+\dots+f_{0}(p)\right)$$
$$g_-(t) :=  (1-t)^{n+1}\sum _{m=1}^{\infty}   t^{m} \left(f_{-1}(p) \frac{1}{m}+\dots+f_{r}(p)\frac{1}{m^{|r|}}\right).$$
and  there exists a positive integer  $k_{0}$  such that $f_{-1}(p)=f_{-2}(p)=\dots=f_{-k_{0}+1}(p)=0$ and $f_{-k_{0}}(p)\neq 0$. 
Notice that 
$$g_-^{(n+k_{0})}(t)=  \sum_{h=0}^{|r|-k_{0}}f_{-k_{0}-h} \left(p\right) \psi_{h} (t),$$
(where $\psi_{h} (t)$ is defined by \eqref{psiht})
and, by  Lemma \ref{lemmatec}, 
\begin{equation}\label{contr1}
\lim_{t \rightarrow 1^-} g_-^{(n+k_0)}(t)=-\infty.
\end{equation}
By combining \eqref{decompPhi},  \eqref{contr1} and  the fact that 
$g_+(t)$ has bounded derivatives of all orders (being $g_+(t)=\sum_{k=0}^{n}f_{k} \left(x\right) q_{k}(t) (1-t)^{n-k}$,  where $q_{k}(t)$ is the  polynomial given by \eqref{pol}) we deduce that $\phi^{(n+k_0)}(t)$ is unbounded in contrast with Lemma \ref{boundedszego}. 

Let now $p_{0} \in M$ and  $e:U \rightarrow L^{*}$ be a local trivialization on a neighborhood of $p_{0}$.  Consider  the coordinate system
$$
v(t,\theta, p) = \sqrt{\frac{t}{h(p)}} e^{i \theta}e(p),
$$
where  $h(p)=h^{*}(e(p),e(p))$ (and hence $h^{*} \left(v(t,\theta, p),\,v(t,\theta, p)\right) =t$) 
By \eqref{szegokempf} and $h^{*} \left(v(t,\theta, p),\,v(t,\theta, p)\right) =t$, one has
\begin{equation}
\phi \left(t\right) = \sum _{m=0}^{\infty}  (1-t)^{n+1}t^{m}\,T_{mg}(p_0)=\rho \left(t\right)^{n+1} {\mathcal S}(v(t,\theta, p_0)) 
\end{equation}
Therefore, by inserting  $g_-(t) = 0$ in (\ref{decompPhi}) one obtains that  $\phi:D \rightarrow \R$ is the restriction of a smooth function on $\ov {D}$ and  by (\ref{fefferman}) one deduces that the log-term of the Szeg\"{o} kernel $\S(v)$ of $D$ must vanish, concluding the proof of the theorem.
\end{proof}

\section{Radial projectively induced cscK metrics with $a_3=0$}\label{sec2}
We first recall the classification of radial \K \ metrics with constant scalar curvature  proved in \cite{LSZ}. Let  $U\subset\C^n$  be a complex domain (not necessarily containing the origin of $\C^n$) 
endowed with a \K\ form $\omega =\frac{i}{2\pi}\partial\bar\partial \Phi$
with radial  potential $\Phi:U\rightarrow \R$, i.e.
$$\Phi(z) = f(r), \ \ r=|z|^2=|z_1|^2+\cdots +|z_n|^2 \in \tilde U:= \{ r = |z|^2, z \in U \}.$$

These metrics can be studied by rewriting everything in terms of the function $\psi(y)$ introduced in the proof of Theorem 2.1 in \cite{LSZ}, i.e. more precisely $F(t)=f(e^t)$, $y = F'(t)$, $\psi(y) = F''(t)$.

In particular, by assuming that $g$ is  cscK,  one shows after a long but straightforward calculation (see the proof of Theorem 2.1 in \cite{LSZ}) that $\psi$ has the form

\begin{equation}\label{psia1const}
\psi(y) = Ay^2 + y  + \frac{B}{y^{n-2}} + \frac{C}{y^{n-1}},
\end{equation}
where  $A$, $B$ and $C$ are constants and the scalar curvature is equal to  $-An(n+1)$. 

\begin{remark}\label{remarkorigin}\rm
Assume  $n=2$.  If we set 
$z=^{\  t}\!\!(z_1, z_2)$ then one easily sees that the matrix of the metric $g$ (still denoted by $g$) reads as:
$$g=\frac{F^{''}-F'}{e^{2t}}z {^{t}\bar z}+\frac{F'}{e^t}I,$$
where $I$ is the $2\times 2$ identity matrix, whose (positive) eigenvalues are $\frac{F'}{e^t}$ and 
$\frac{F''}{e^t}$. So, if we further assume that $\Phi$ is defined at the origin, we get
$$\lim_{t\rightarrow-\infty}F'=\lim_{t\rightarrow-\infty} F''=0,$$
forcing $B=C=0$ in  (\ref{psia1const}). In this case the solution of $\psi(y) = Ay^2 + y$ are the flat, the Fubini-Study and the hyperbolic metric if $A=0$, $A<0$ and $A>0$, respectively (cfr. \eqref{A=01}, \eqref{An0FS} and \eqref{An0hyp} below).
\end{remark}

In the proof of Theorem 1.1 in \cite{LSZ} it is shown that $n=2$ and $a_3 = 0$ (where $a_3$ is the third coefficient of TYCZ expansion of the Kempf distortion function) if and only if $C = 0$, so (\ref{psia1const}) reduces under these assumptions to 

\begin{equation}\label{psia1constRed}
\psi(y) = Ay^2 + y  + B
\end{equation}

\begin{remark}\label{a2vanish}\rm
For a cscK radial metric the condition $a_3=0$  is equivalent to $a_2=0$ (see \cite{LSZ}).
This fact will be used in the proof of Theorem \ref{class2}.
\end{remark}

Hence the classification of radial cscK metrics with $a_3=0$ reduces to integrating equation (\ref{psia1constRed}) (recall that $\psi = y'$, where the derivative is meant with respect to $t = \log r$) in the cases $A=0$ and $A \neq 0$. In the latter  we further distinguish the three cases where the equation $Ay^2 + y  + B =0$ has no real solutions, only one real solution or two real solutions, and the sign of these solutions. Let us briefly recall the result of such classification. In order to keep the same notation used in \cite{LSZ}, we will rewrite $\psi$ in terms of real parameters $\lambda, \mu, \xi, \zeta > 0$, $0<\zeta<1$, $\kappa \in \R$ (the exact relation with $A,B$ is not necessary for our purposes).

When $a_1 =0$ (namely vanishing scalar curvature or, equivalently, $A=0$) we have the following three cases:

\begin{equation}\label{A=01}
\psi (y)=   y  
\end{equation}
which corresponds to the flat metric $g_0$ on $U\subseteq\C^2$;
\begin{equation}\label{A=02}
\psi(y) =   y - \lambda  
\end{equation}
which integrates as $F'(t) = \mu e^t + \lambda$, is defined on $r = e^t > 0$ and is (a multiple of) the Simanca metric (\ref{simancaprimavolta}) on $U\subseteq\C^2\setminus\{0\}$;
\begin{equation}\label{A=03}
\psi (y)=   y + \lambda  
\end{equation}
which integrates as $F'(t) = \mu e^t - \lambda$ and is defined on $r = e^t > \frac{\lambda}{\mu}$. Notice that $F' \rightarrow 0$ when $r \rightarrow \frac{\lambda}{\mu}$.

When $a_1 \neq 0$  (equivalently, $A \neq 0$) we have the following eight cases [\eqref{An0FS}--\eqref{An0viii}]:
\begin{equation}\label{An0FS}
\psi (y)=   \frac{1}{\mu}y(\mu - y)  
\end{equation}
which integrates as $y = F'(t) = \frac{\mu e^t}{1 + e^t}$ and corresponds to the multiple  $\mu \omega_{FS} = \mu i \partial \bar \partial \log(1 + |z|^2)$ of the Fubini-Study metric on $U\subseteq\C^2\subset \C P^2$;
\begin{equation}\label{An0hyp}
\psi(y) =   \frac{1}{\mu}y(\mu + y)  
\end{equation}
which integrates as $y = F'(t) = \frac{\mu e^t}{1 - e^t}$ and corresponds to the multiple  $\mu \omega_{hyp} = -\mu i \partial \bar \partial \log(1 - |z|^2)$ of the hyperbolic metric on $U\subseteq\C H^2$;
\begin{equation}\label{An0iii}
\psi(y) =  \left[ \left( \frac{1}{\mu}y  + \frac{1}{2} \right)^2 + \lambda^2 \right]
\end{equation}
which is easily seen to integrate  as $y = F'(t) = \mu \left[ \lambda \tan(\lambda t + \kappa) - \frac{1}{2} \right]$
with maximal interval of definition given by $h \pi + \arctan \left( \frac{1}{2 \lambda} \right) < \lambda t + \kappa < \frac{2h+1}{2} \pi$.
Notice that $F' \rightarrow 0$ when $\lambda t + \kappa \rightarrow h \pi + \arctan \left( \frac{1}{2 \lambda} \right)$;
\begin{equation}\label{An0iv}
\psi(y) = \frac{1}{\mu} \left( y  - \frac{1-\zeta}{2} \right) \left( y  - \frac{1-\zeta}{2} + \mu \right) 
\end{equation}
which is easily seen to integrate  as $y = F'(t) = -\mu \left[ \frac{- \xi \zeta e^{\zeta t}}{1 - \xi e^{\zeta t}} + \frac{1-\zeta}{2} \right]$
with maximal interval of definition given by  $\frac{1-\zeta}{1+\zeta}  < \xi e^{\zeta t} < 1$.
Notice that $F' \rightarrow 0$ when $\xi e^{\zeta t} \rightarrow \frac{1-\zeta}{1+\zeta}$;
\begin{equation}\label{An0v}
\psi (y)= \frac{1}{\mu} \left( y  - \frac{\mu \lambda}{2} \right) \left( y  + \frac{\mu \lambda}{2} + \mu \right) 
\end{equation}
which is easily seen to integrate  as $y = F'(t) = -\mu \left[ \frac{- \xi (\lambda+1) e^{(\lambda+1) t}}{1 - \xi e^{(\lambda+1) t}} - \frac{\lambda}{2} \right]$
with maximal interval of definition given by $0  < \xi e^{(\lambda+1) t} < 1$.
Notice that $\frac{\mu \lambda}{2} < F' < + \infty$:
\begin{equation}\label{An0vi}
\psi (y)= -\frac{1}{(\lambda+1)\mu}\left( y  + \frac{\mu \lambda}{2} \right)   \left( y  - \frac{\mu \lambda}{2} - \mu \right) 
\end{equation}
which is easily seen to integrate  as $y = F'(t) = \mu \left[ \frac{- \xi (\lambda+1) e^{-(\lambda+1) t}}{1 + \xi e^{-(\lambda+1) t}} + \frac{2 + \lambda}{2} \right]$
with maximal interval of definition given by $\xi e^{(\lambda+1) t} > \frac{\lambda}{2 + \lambda}$.
Notice that $F' \rightarrow 0$ when $\xi e^{(\lambda+1) t} \rightarrow \frac{\lambda}{2+\lambda}$;
\begin{equation}\label{An0vii}
\psi (y)= -\frac{1}{\mu}\left( y  - \frac{\mu (1 + \zeta)}{2} \right)   \left( y  - \frac{\mu (1-\zeta)}{2} \right) 
\end{equation}
which is easily seen to integrate  as $y = F'(t) = \mu \left[ \frac{- \zeta e^{-\zeta t}}{1 + e^{-\zeta t}} + \frac{1 + \zeta}{2} \right]$
with maximal interval of definition $\R \setminus \{0 \}$. Notice that $\frac{\mu (1-\zeta)}{2} < y < \frac{\mu (1+\zeta)}{2}$;
\begin{equation}\label{An0viii}
\psi (y)= \frac{1}{\mu}\left( y  + \frac{\mu}{2} \right)^2  
\end{equation}
which is easily seen to integrate  as $y = F'(t) = \mu \left[ \frac{1}{k - t} - \frac{1}{2} \right]$
with maximal interval of definition $k-2 < t < k$. Notice that $F' \rightarrow 0$ when $t \rightarrow k-2$.

Metrics (\ref{An0iii})-(\ref{An0viii}) correspond respectively to cases 11a, 6, 7, 8, 9, 10a of Theorem 2.1 in \cite{LSZ}.

In \cite{LSZ} the first and the third author together with F. Salis proved that the flat metric and the Simanca metric, i.e. the cases (\ref{A=01}) and (\ref{A=02}) above, are the only radial projectively induced metrics with $a_1 = a_3 = 0$.
\begin{remark}\label{projind}\rm
For the hyperbolic metric $g_{hyp}$ on $\C H^n$ and the Fubini-Study metric $g_{FS}$ on $\C^n$, i.e. the cases (\ref{An0hyp}) and (\ref{An0FS}) above, one has that $\mu g_{hyp}$ admits an {\em injective} K\"ahler immersion into $\C P^{\infty}$ for any $\mu >0$, while $\mu g_{FS}$ admits an {\em injective}   K\"ahler immersion into $\C P^{N}$ if and only if $\mu$ is an integer (the reader is referred to \cite{Cal} for an explicit descriptions of these maps).
\end{remark}
Among the other cases above, it is easy to see that the metric (\ref{An0v})  is projectivel{y induced provided $\lambda \in \z$ and $\frac{\lambda \mu}{2} \in \z$. Indeed, an explicit potential of this metric is given by 
\begin{equation}\label{explmetr7}
\hat\Phi=   \log \frac{(|z_1|^2 + |z_2|^2)^{\frac{\mu \lambda}{2}}}{[1 - \xi (|z_1|^2 + |z_2|^2)^{(\lambda+1)}]^{\mu}} 
\end{equation}

and, by using $\frac{1}{(1-x)^{\mu}} = \sum_{i=0}^{\infty} \frac{\mu(\mu+1) \cdots (\mu + i -1)}{i!} x^i$ one has

\begin{equation*}
\begin{split}
e^{\hat\Phi} = &(|z_1|^2 + |z_2|^2)^{\frac{\mu \lambda}{2}} \sum_{i=0}^{\infty} \frac{\mu(\mu+1) \cdots (\mu + i -1)}{i!} \xi^i (|z_1|^2 + |z_2|^2)^{(\lambda+1)i} \\ 
=&  \sum_{i=0}^{\infty} \frac{\mu(\mu+1) \cdots (\mu + i -1)}{i!} \xi^i (|z_1|^2 + |z_2|^2)^{(\lambda+1)i + \frac{\mu \lambda}{2}} \\
=&\sum_{i=0}^{\infty} \sum_{j=0}^{(\lambda+1)i + \frac{\mu \lambda}{2}} \frac{\mu(\mu+1) \cdots (\mu + i -1)}{i!} \xi^i {(\lambda+1)i + \frac{\mu \lambda}{2} \choose j} |z_1|^{2j} |z_2|^{2(\lambda+1)i + \mu \lambda - 2j}.
\end{split}
\end{equation*}
Then,
\begin{equation}\label{projemb}
(z_1, z_2) \mapsto \left[ \cdots , \sqrt{ \frac{\mu(\mu+1) \cdots (\mu + i -1)}{i!} \xi^i {(\lambda+1)i + \frac{\mu \lambda}{2} \choose j}} z_1^{j} z_2^{k}, \cdots\right]
\end{equation}
for $i=0, \dots, \infty$ and  $j+k=(\lambda+1)i + \frac{\mu \lambda}{2}$
gives the desired projective immersion.

\begin{remark}
\rm Notice that the metric given by potential (\ref{explmetr7}) is not Einstein for any values of $\lambda, \mu, \xi > 0$. Indeed, for a \K\ metric $g$ on a 2-dimensional manifold with \K\ form $\omega = \frac{i}{2\pi} \partial \bar \partial \Phi$  given by a radial potential $\Phi(z) = f(r), r= |z_1|^2 + |z_2|^2$, we have

$$\det(g) = \det \left(  \begin{array}{cc}
f' + f'' \cdot |z_1|^2 & f'' \bar z_1 z_2   \\
f'' \bar z_2 z_1   & f' + f'' \cdot |z_2|^2 
\end{array}\right) = f'^2 + f' f'' r = f' (r f')'.$$

Then, after a straightforward computation with $\Phi$ given by (\ref{explmetr7}) we get

\begin{equation}\label{detg}
\det(g) = \frac{\xi \mu^2 (\lambda + 1)^2}{2} \frac{r^{\lambda-1}[\lambda + \xi(\lambda + 2) r^{\lambda+1}]}{(1 - \xi r^{\lambda+1})^3}
\end{equation}
and one immediately sees that the metric is not Einstein by comparing
$\log \det(g) $ with (\ref{explmetr7}).
\end{remark}

The following proposition, interesting on its own sake,  shows that the only radial cscK projectively induced metrics with $a_3=0$ are those just described. It could be interesting to classify all the radial projectively induced cscK metrics without the assumption of the vanishing of  $a_3$ (the reader is referred to \cite{LSZconj} for the classification of  radial  projectively induced Ricci flat \K\ metrics). 
\begin{prop}\label{radprojindcscka3=0}
 Let  $U\subset\C^n$  be a complex domain on which is defined a radial \K \ metric $g$  given by a radial  potential $\Phi:U\rightarrow \R$. Assume that $g$ is a cscK metric and $a_3=0$. Then $g$ is projectively induced if and 
 only if we are in the cases (\ref{A=01}), (\ref{A=02}) with $\lambda \in \z$, (\ref{An0FS}) with $\mu \in \z$, (\ref{An0hyp}) for any $\mu$  and (\ref{An0v}) with $\lambda, \frac{\lambda \mu}{2} \in \z$, of the above classification.
\end{prop}

In order to prove the proposition we need three  lemmata.
\begin{lem}\label{LemmaLSZ}
Let  $U\subset\C^n$  be a complex domain 
endowed with a \K\ metric $g$  whose associated \K\ form  $\omega =\frac{i}{2\pi}\partial\bar\partial \Phi$
hass radial  potential $\Phi:U\rightarrow \R$, i.e.
$\Phi(z)=f(r)$, $r=|z|^2$.
 If there exist  $r \in\tilde U:= \{ r = |z|^2, z \in U \}$ and $h\in\N$ such that  
\begin{equation}\label{gh}
g_h(r)=\frac{d^{h}e^{f(r)}}{dr^{h}}<0
\end{equation}
then  $g$ is not projectively induced.
\end{lem}
\begin{proof}
See Lemma 3.1 in  \cite{LSZconj} for a proof.
\end{proof}

\begin{lem}\label{lemmaprimocriterio}
Let $g$ be a radial \K\  metric as above and  let $\psi(y) = Ay^2 + y + B$ given by  
\eqref{psia1constRed}.
Assume $g$ is projectively induced and $0$ is a limit point in the domain of definition of $\psi$. 
Then $B=0$ (i.e. $\psi$ has 0 as root).
\end{lem}

\begin{proof}
By taking $h=1,2,3$ in (\ref{gh}), one gets in particular that if $g$ is projectively induced then
\begin{equation}\label{gh123}
\left. \begin{array}{c}
f' \geq 0  \\
f'^2 + f'' \geq 0 \\
f''' + 3 f' f'' + f'^3 \geq 0 \\
\end{array} \right.
\end{equation}

(the derivatives are meant with respect to $r$). These conditions can be rewritten in terms of the function $\psi(y)$ introduced above. More precisely, since $f(r) = \Phi(t)|_{t = \log r}$, $y = \Phi'(t)$, $\psi(y) = \Phi''(t)$, we have $f' = \frac{y}{r}$, $f'' = \frac{\psi - y}{r^2}$,  $f''' = \frac{\psi'\psi - 3 \psi + 2 y}{r^3}$ and (\ref{gh123}) rewrite

\begin{equation}\label{gh123rew}
\left. \begin{array}{c}
y \geq 0  \\
\psi - y + y^2 \geq 0 \\
3y\psi+\psi' \psi - 3 \psi + 2 y - 3 y^2 + y^3  \geq 0 \\
\end{array} \right.
\end{equation}

Now, by replacing $\psi(y) = Ay^2 + y  + B$ in (\ref{gh123rew}) one gets

\begin{equation}\label{gh123rewRepl}
\left. \begin{array}{c}
y \geq 0  \\
(A+1)y^2 + B \geq 0 \\
(2 A y +3y- 2)(A y^2 + y + B)+ 2 y - 3 y^2 + y^3 \geq 0 \\
\end{array} \right.
\end{equation}
From the second and the third condition one immediately deduces that if the metric is such that in the interval of definition one can let $y = \Phi'(t)$ tend to zero, then it must be $B=0$, as claimed.
\end{proof}
\begin{lem}\label{radiceintera}
Let $g$ be a radial \K\  metric as above and  let $\psi(y) = Ay^2 + y + B$ given by  
\eqref{psia1constRed}.
Assume $g$ is projectively induced and   $y_0$ is a limit point in the domain of $\psi$.
If $y_0$ is a positive  root of 
$\psi$ then $y_0 \in \z$.
\end{lem}
\begin{proof}
For every $k \geq 1$ one can prove by induction on $k$ that 
\begin{equation}\label{formulaiND}
e^{-f}\frac{d^k e^f}{d r^k}  = \frac{\psi P_k + y(y-1)(y-2) \cdots (y-k+1)}{r^k},
\end{equation}
where $P_k$ is a polynomial in $y$.
Indeed, for $k=1$ one has
$$e^{-f}\frac{d e^f}{d r} = f' = \frac{F'(t)}{r} = \frac{y}{r}$$
that is (\ref{formulaiND}) with $P_0=0$. Now, assuming by induction that (\ref{formulaiND}) is true for some $k$, we claim that it is true for $k+1$. Indeed, one has
\begin{equation*}
\begin{split}
\frac{d^{k+1} e^f}{d r^{k+1}} 
=&\frac{d}{dy} \left[ \psi P_k + y(y-1)(y-2) \cdots (y-k+1) \right] \frac{dy}{dr} \cdot \frac{e^f}{r^k}+ \\
+&\left[ \psi P_k + y(y-1)(y-2) \cdots (y-k+1) \right] \frac{d}{dr} \left( \frac{e^f}{r^k} \right)
\end{split}
\end{equation*}
and the claim follows by using 
$$\frac{dy}{dr} = \frac{d}{dr}F'(t) = \frac{F''(t)}{r} = \frac{\psi}{r},$$
$$\frac{d}{dr} \left( \frac{e^f}{r^k} \right) = \frac{rf' - k}{r^{k+1}} e^f = \frac{y- k}{r^{k+1}} e^f$$
and the fact that $\psi$ is a polynomial.
Now, notice that from (\ref{formulaiND}) it follows that if $y_0$ is a root of $\psi$, then
\begin{equation}\label{formulaiNDroot}
e^{-f}\frac{d^k e^f}{d r^k} _{|_{y=y_0}} = \frac{ y_0(y_0-1)(y_0-2) \cdots (y_0-k+1)}{r^k}
\end{equation}
Assume $y_0 \notin \z$.
By continuity, in a neighbourhood of $y=y_0$ one has that $\frac{d^k e^f}{d r^k}$ has the same sign of $y_0(y_0-1)(y_0-2) \cdots (y_0-k+1)$ which is strictly negative for $k = [y_0]+2$ (where $[y_0]$ denotes the integer part of $y_0$). Hence, by Lemma \ref{LemmaLSZ}, the metric is not projectively induced.
\end{proof}

\begin{proof}[Proof of Proposition \ref{radprojindcscka3=0}]
By the discussion before the statement of the proposition, we are left to prove that the metrics corresponding to cases (\ref{A=03}), (\ref{An0iii}), (\ref{An0iv}), (\ref{An0vi}), (\ref{An0vii}) and (\ref{An0viii})
are not projectively induced for any values of $\lambda, \xi, \mu$, and that if $\lambda \notin \z$ or $\frac{\lambda \mu}{2} \notin \z$ then the metric (\ref{An0v}) is not projectively induced.
 
By Lemma \ref{lemmaprimocriterio}  one immediately sees that the metrics (\ref{A=03}), (\ref{An0iii}), (\ref{An0iv}), (\ref{An0vi}), and (\ref{An0viii}) are not projectively induced for any values of the parameters.

For the remaining cases (\ref{An0v}) and (\ref{An0vii}), we cannot use the same argument since we have respectively $\frac{\mu \lambda}{2} < y < + \infty$ and $\frac{\mu (1-\zeta)}{2} < y < \frac{\mu (1+\zeta)}{2}$ and so $0$ is not a limit point in the domain of definition of $\psi$, so we will use another approach. More precisely, we will take the explicit expressions of the potentials of these metrics, which are respectively (see the statement of Theorem 2.1 in \cite{LSZ})

\begin{equation}\label{explpotcasev}
f(r) =\log \frac{r^{\frac{\mu \lambda}{2}}}{(1 - \xi r^{\lambda+1})^{\mu}} 
\end{equation}
and
\begin{equation}\label{explpotcasevii}
f(r) =\log [r^{\frac{\mu(1+\zeta)}{2}}(1 + r^{-\zeta})^{\mu}]
\end{equation}
and we will apply the criterion given in Lemma \ref{LemmaLSZ}. 

Let us begin from case (\ref{An0vii}): by Lemma \ref{radiceintera}, if one of the two roots 

$$  k = \frac{\mu (1-\zeta)}{2}, \ \ l = \frac{\mu (1 + \zeta)}{2}$$

\noindent of $\psi$ is not an integer, then the metric is not projectively induced. Assume thus that $k ,l \in \z^+$, which implies also $k+l = \mu \in \z$. Then, by (\ref{explpotcasevii}) we have

$$e^{f(r)} = r^l (1 + r^{-\zeta})^{\mu} = \sum_{s=0}^{\mu} {\mu \choose s} r^{l - \zeta s} =  \sum_{s=0}^{\mu} {\mu \choose s} r^{k + (\mu - s) \zeta}.$$

\noindent By a straight calculation one sees that, for $k_0 = k+2$ one has

$$\frac{d^{k_0}}{dr^{k_0}} e^{f(r)} = r^{\zeta - 2} (c_0 + c_1 r^{\zeta} + c_2 r^{2 \zeta} + \cdots + c_{\mu-1} r^{(\mu-1)\zeta}),$$
for suitable constants $c_j$, with
$$c_0 = \mu (\zeta + k) \cdots (\zeta+1) \zeta (\zeta -1)$$
is negative since $0 < \zeta < 1$. This implies that $\frac{d^{k_0}}{dr^{k_0}} e^{f(r)} \rightarrow - \infty$ for $r \rightarrow 0^+$ and proves that the metric is not projectively induced for any values of the parameters.

For the last case (\ref{An0v}), we first notice that, by Lemma \ref{radiceintera}, if the root $\frac{\mu \lambda}{2}$ of $\psi$ is not an integer then the metric is not projectively induced. We are then left to show that when $\frac{\mu \lambda}{2} \in \z$ and $\lambda \notin \z$ then the metric is not projectively induced. By (\ref{explpotcasev}) one has

$$e^{f(r)} = \sum_{i=0}^{\infty} \frac{\mu(\mu+1) \cdots (\mu + i -1)}{i!} \xi^i r^{(\lambda+1)i + \frac{\mu \lambda}{2}}$$

Then, by a straightforward computation, one sees that, for $k_0 = \frac{\mu \lambda}{2} + [\lambda] + 3$ (where $[\lambda]$ denotes the integer part of $\lambda$)

$$\frac{d^{k_0}}{dr^{k_0}} e^{f(r)} = r^{\lambda- [\lambda] - 2} (c_0 + c_1 r^{\lambda+1} + c_2 r^{2(\lambda+1)} + \cdots)$$

where

$$c_0 = \mu \xi \left( \lambda + \frac{\mu \lambda}{2} + 1 \right) \left( \lambda + \frac{\mu \lambda}{2} \right) \cdots   (\lambda - [\lambda]) (\lambda - [\lambda]-1)$$

\noindent is strictly negative since we are assuming $\lambda \notin \z$. This implies that $\frac{d^{k_0}}{dr^{k_0}} e^{f(r)} \rightarrow - \infty$ for $r \rightarrow 0^+$ and concludes the proof.

\end{proof}

\section{The proofs of Theorem \ref{class1} and Theorem \ref{class2}}\label{sec3}

\begin{proof}[Proof of Theorem \ref{class1}]
By Theorem \ref{mainteor}, the finiteness of TYCZ expansion implies that the Kempf distortion function reduces to the  polynomial 
$T_{mg}=m+a_1$. This forces  $a_2=0$ and hence, using again the fact that $M$ is a complex curve, one deduces by \eqref{coefflu} that $a_1=\frac{1}{2}{\rm scal}_g=const$, namely  the metric $g$ is a cscK metric. Notice that, by completeness, if  $(M, g)$ were  simply-connected then  one would  deduce that  it is a one-dimensional  complex space form (a), (b) and (c), where $\lambda$ is a positive integer  (we are also using the integrality of the \K\ form $\omega$ associated to $g$ to obtain the integrality of $\lambda$).
 Hence, in order to prove the theorem, 
we are reduced to show that  $M$ is simply-connected. Assume, by contradiction, that $M$ is not simply connected and let $p:(\tilde M, \tilde g)\rightarrow (M, g)$ be the universal covering map (which is a non-injective \K\ immersion satisfying $p^*g=\tilde g$). Then $(\tilde M, \tilde g)$ would be one of the three one-dimensional complex space forms (a), (b), (c),  and hence there exists an {\em injective} full  \K\ immersion $\psi :\tilde M\rightarrow \CP^N$ (see Remark \ref{projind} above).  Since $T_{g}=1+a_1$ is constant one deduce  (see \eqref{obstr}) that the coherent states 
map $\varphi_1:M\rightarrow \CP^{N(1)}$ is a full  \K\ immersion. Hence the holomorphic map $\varphi_1\circ p:\tilde M\rightarrow \CP^{N(1)}$  satisfies   $(\varphi_1\circ p)^*g_{FS}=p^*\varphi_1^*g_{FS}=\tilde g$. By the celebrated Calabi's rigidity theorem
\cite{Cal}
$N(1)=N$ and  there exists a unitary transformation $U$ of $\CP^N$ such that $U\circ \psi= \varphi_1\circ p$.
This forces $\varphi_1\circ p$ and hence $p$ to be injective, yelding the desired contradiction.
\end{proof}

Finally, we prove Theorem \ref{class2}.

\begin{proof}[Proof of Theorem \ref{class2}]
Combining the assumptions  with Theorem \ref{mainteor} and Remark \ref{a2vanish}  one gets that, for some constant $a_1$, the Kempf distortion function associated to $(M, g)$.  is given by
$T_{mg}=m^2+a_1m$.
Therefore the metric $g$ is forced to be balanced for all $m$ (or equivalently $(L, h)$ is a regular quantization).
Recall that a balanced metric is automatically projectively induced and, as we have already pointed out in the Introduction, (i), (ii), (iii) and (iv) in Theorem \ref{class2} all admit an open and dense subset with a cscK metric with radial potential with finite TYCZ expansion. Thus,  by using Proposition \ref{radprojindcscka3=0}, we are left to show that the metric of case (\ref{An0v}) of the classification in Section 3, given by potential (\ref{explmetr7}), does not admit a regular quantization for $\xi, \lambda,  \mu >0$ with $\lambda, \frac{\lambda \mu}{2} \in \z$. In order to do that, recall that by \eqref{obstr} 
this happens if and only if
\begin{equation}\label{defbalanced}
\frac{i}{2\pi} \partial \bar \partial \log \sum_{j}^{} |s_j(z)|^2 = \frac{i}{2\pi} \partial \bar \partial \hat\Phi
\end{equation}
where $\{ s_j \}$ is an orthonormal basis of the space $\mathcal{H}_{\mu, \lambda, \xi}$ of holomorphic functions $s = s(z)$ on the domain of definition of the metric
$$U= \left\{ r = |z_1|^2 + |z_2|^2 \ | \ r < \left( \frac{1}{\xi} \right)^{\frac{1}{\lambda+1}} \right\}$$
which are bounded with respect to the norm
\begin{equation}\label{norm}
\| s \|^2_{h_{\mu}} = \int_Uh_{\mu}(z) |s(z)|^2 dv(z)
\end{equation}
endowed with the hermitian product\footnote{Notice that we are using the fact that $U$ is dense in $M$ in order  to integrate on $U$.} $\langle s, t \rangle_{h_{\mu}} = \int_U h_{\mu}(z) s(z) \overline{t(z)} dv(z)$
(cf. \eqref{hilbertspace} and \eqref{kempfnonc} in the introduction), where

$$h_{\mu}(z) =e^{-\hat\Phi(z)} = \frac{[1 - \xi r^{(\lambda+1)}]^{\mu}}{r^{\frac{\mu \lambda}{2}}}, \ \ r = |z_1|^2 + |z_2|^2$$

\noindent and $dv(z) = \left( \frac{i}{2\pi} \right)^2 \det(g) dz_1 \wedge d \bar z_1\wedge dz_2 \wedge d \bar z_2$ is the volume form.

Now, take $s(z) = z_1^j z_2^k$. By passing to polar coordinates $z_1 = \rho_1 e^{i \theta_1}$, $z_2 = \rho_2 e^{i \theta_2}$, and using (\ref{detg}) we get

\begin{equation*}
\begin{split}
\| s \|^2_{h_{\mu}} = & \int_U  |z_1|^{2j} |z_2|^{2k} h_{\mu} d v(z) \\ 
=&
 2 \xi \mu^2 (\lambda + 1)^2 \int  \rho_1^{2j+1} \rho_2^{2k+1} \frac{(1 - \xi (\rho_1^2 + \rho_2^2)^{\lambda+1})^{\mu}}{(\rho_1^2 + \rho_2^2)^{\frac{\mu \lambda}{2} }}  \frac{(\rho_1^2 + \rho_2^2)^{\lambda-1}[\lambda + \xi(\lambda + 2) (\rho_1^2 + \rho_2^2)^{\lambda+1}]}{(1 - \xi (\rho_1^2 + \rho_2^2)^{\lambda+1})^3} d \rho_1 d \rho_2,
\end{split}
\end{equation*}

where we are integrating on $\rho_1^2+\rho_2^2<(\frac{1}{\xi})^\frac{1}{\lambda+1}$.

Now by setting $\rho = \sqrt{\rho_1^2 + \rho_2^2}$ we can make the substitution $\rho_1 = \rho \cos \theta, \rho_2 = \rho \sin \theta$, $0 < \rho < \infty$, $0 < \theta < \frac{\pi}{2}$, 
and using
$$\int_0^ \frac{\pi}{2} (\cos \theta)^{2j+1} (\sin \theta)^{2k+1}=\frac{j!k!}{2(j+k+1)!}$$
the previous  integral becomes
\begin{equation}\label{integral1}
 \xi \mu^2 (\lambda + 1)^2 \frac{j!k!}{(j+k+1)!} \int_0^{\left( \frac{1}{\xi} \right)^{\frac{1}{2(\lambda+1)}} }  \rho^{2j+2k+2 \lambda - \mu \lambda +1} \frac{\lambda + \xi(\lambda + 2) \rho^{2(\lambda+1)}}{(1 - \xi \rho^{2(\lambda+1)})^{3-  \mu}} d \rho.
\end{equation}

Let us make the change of variable 

$$x = \xi \rho^{2(\lambda+1)}, \ \ dx = 2\xi(\lambda+1) \rho^{2 \lambda +1} d \rho$$

\noindent and (\ref{integral1}) rewrites

\begin{equation}\label{integral2}
 \xi^{-\frac{2j+2k-\mu \lambda}{2(\lambda+1)}} \mu^2 (\lambda + 1) \frac{j!k!}{(j+k+1)!} \int_0^{1 } x^{\frac{2j+2k-\mu \lambda}{2(\lambda+1)}} \frac{\lambda + (\lambda + 2)x }{(1 - x)^{3-  \mu}} d x
\end{equation}

and then one easily sees that it converges if and only if $\mu > 2$ and $ \frac{2j+2k-\mu \lambda}{2(\lambda+1)} > -1$, i.e.

\begin{equation}\label{conditionsConv}
j+k >  \frac{\mu \lambda}{2} -(\lambda+1).
\end{equation}

This is the condition for a monomial $z_1^j z_2^k$ to belong to the space $\mathcal{H}_{\mu, \lambda, \xi}$. Since by radiality it is easy to see that the monomials $z_1^j z_2^k$ are pairwise orthogonal, we see that $\{ z_1^j z_2^k \}_{j+k >  \frac{\mu \lambda}{2} -(\lambda+1)}$ form a complete orthogonal basis of in $\mathcal{H}_{\mu, \lambda, \xi}$, so the condition (\ref{defbalanced}) for the metric to be balanced can be rewritten

\begin{equation}\label{defbalanced2}
\frac{i}{2\pi} \partial \bar \partial \log \left[\sum_{j+k > \frac{\mu \lambda}{2} -(\lambda+1)} \frac{|z_1|^{2j} |z_2|^{2k}}{\|z_1^j z_2^k\|^2_{h_{\mu}}}\right] = \frac{i}{2\pi} \partial \bar \partial \hat\Phi
\end{equation}

This means that there exists a holomorphic function $f$ such that

$$ \log\left[ \sum_{j+k > \frac{\mu \lambda}{2} -(\lambda+1)} \frac{|z_1|^{2j} |z_2|^{2k}}{\|z_1^j z_2^k\|^2_{h_{\mu}}}\right] = \hat\Phi + Re(f).$$

By radiality, $f$ is forced to be constant and we can rewrite this condition as

\begin{equation}\label{defbalanced3}
\sum_{j+k > \frac{\mu \lambda}{2} -(\lambda+1)} \frac{|z_1|^{2j} |z_2|^{2k}}{\|z_1^j z_2^k\|^2_{h_{\mu}}}= C e^{\hat\Phi} = C \frac{(|z_1|^2 + |z_2|^2)^{\frac{\mu \lambda}{2}}}{[1 - \xi (|z_1|^2 + |z_2|^2)^{(\lambda+1)}]^{\mu}}
\end{equation}

\noindent for some $C>0$.

Now, we notice that since we are assuming $\lambda > 0$, then condition (\ref{conditionsConv}) is fulfilled for $j+k = \frac{\mu \lambda}{2} -1$ (recall that $\lambda, \mu > 0$ and that we are assuming that $\frac{\mu \lambda}{2} \in \z$, otherwise the metric is not projectively induced). But it is easy to see that the Taylor expansion of the right-hand side of (\ref{defbalanced3}) does not contain the term $|z_1|^{2j} |z_2|^{2k}$ for  $j+k = \frac{\mu \lambda}{2} -1$, so (\ref{defbalanced3})  cannot be satisfied and the metric is not balanced. This concludes the proof of  the theorem.\end{proof}

\end{document}